\documentclass[11pt, psamsfonts, reqno]{amsart}
\usepackage{amsmath, amsfonts, amsthm, amssymb}
\usepackage{graphicx}
\usepackage{float}
\usepackage{verbatim}
\usepackage{hyperref}

\numberwithin{equation}{section}

\newtheorem{thm}{Theorem}[section]
\newtheorem{lma}[thm]{Lemma}

\renewcommand{\le}{\leqslant}
\renewcommand{\geq}{\geqslant}
\renewcommand{\leq}{\leqslant}

\def\eps{\varepsilon}
\def\phi{\varphi}

\def\D{{\mathcal D}}

\def\b{C}

\begin{document}

\title{Quantifying inhomogeneity in fractal sets}

\author[J. M. Fraser]{Jonathan M. Fraser}
\address{Jonathan M. Fraser\\ 
School of Mathematics\\ The University of Manchester\\ Manchester\\ M13 9PL\\ UK \\} \email{\href{mailto:jonathan.fraser@manchester.ac.uk}{jonathan.fraser@manchester.ac.uk} }
\urladdr{\url{http://personalpages.manchester.ac.uk/staff/jonathan.fraser/}}

\author[M. Todd]{Mike Todd}
\address{Mike Todd\\ Mathematical Institute\\
University of St Andrews\\
North Haugh\\
St Andrews\\
KY16 9SS\\
Scotland \\} 
\email{\href{mailto:m.todd@st-andrews.ac.uk}{m.todd@st-andrews.ac.uk}}
\urladdr{\url{http://www.mcs.st-and.ac.uk/~miket/}}

\date{\today}

      \maketitle

\begin{abstract}
An inhomogeneous fractal set is one which exhibits different scaling behaviour at different points.   
The Assouad dimension of a set is a quantity which finds the `most difficult location and scale' at which to cover the set and its difference from box dimension can be thought of as a first-level overall measure of how inhomogeneous the set is.  For the next level of analysis, we develop a quantitative theory of inhomogeneity by considering the measure of the set of points around which the set exhibits a given level of inhomogeneity at a certain scale.  For a set  of examples, a family of $(\times m, \times n)$-invariant subsets of the 2-torus, we show that this quantity satisfies a Large Deviations Principle. We compare members of this family, demonstrating how the rate function gives us a deeper understanding of their inhomogeneity.
\\ \\
\emph{Mathematics Subject Classification}  37C45, 60F10 (primary); 60F05, 28A80 (secondary).\\ 
\emph{Key words and phrases}: large deviations, Assouad dimension, box dimension, self-affine carpet.
\end{abstract}

\section{Introduction}
If a set is self-similar then it scales in the same way in all directions however far we zoom in, and wherever in the set we zoom in.  In this case notions of Hausdorff dimension, box dimension etc give good characterisations of local and global scaling behaviour.  If, on the other hand, there are different amounts of scaling in different directions, and this difference varies over the set, then this inhomogeneity is not really seen by the standard dimensions.  This phenomenon is present in a number of interesting classes of sets which are currently attracting a great deal of attention including: self-affine sets, non-conformal repellers, parabolic repellers such as parabolic Julia sets, and some random constructions.  For such sets, Assouad dimension, and in particular its difference from box dimension, gives some idea of this inhomogeneity.  However, this value only gives a very superficial measure of inhomogeneity and so, in order to obtain a detailed understanding of the scaling properties, and their inhomogeneity across the set, we perform the first quantitative study of Assouad dimension.  We will focus on a classic family of examples: certain compact $(\times m, \times n)$-invariant subsets of the 2-torus.

\subsection{Quasi-Assouad dimension}
Let $X$ be a compact metric space, which we will later assume to be the 2-torus, identified with $[0,1]^2$.  We wish to study the dimension of fractal subsets $F$ of $X$ via a family of observables based on location and scale.  More precisely, given a location $x \in  F$ and a scale $R>0$, we wish to describe how `large' $F$ is in a small set $\b(x, R)$ around $x$, where  $\b(x, R)\to \{x\}$ as $R\to 0$.  In standard cases $\b(x, R)$ would be taken to be an $R$-ball, but often, as here, it can be convenient/appropriate to take another type of set,  like a generalised cube for instance.  To this end, for $r>0$, let $N_r(E)$ be the smallest number of sets with diameter no greater than $r$ required to cover $E$, where a collection of sets $\{U_i\}_i$ is a cover of $E$ if $E \subseteq \cup_i U_i$.  We can now describe the size of $\b(x,R) \cap F$ by computing $N_r\big(\b(x,R) \cap F)$ for $r \in (0,R)$ and then taking the extremal exponential growth rate of this quantity as $R/r \to \infty$.  Fix $\varepsilon>0$ and define  the \emph{$\eps$-local Assouad dimension at location $x$ at scale $R$} by
\[
A_F^\varepsilon (x,R) \ :=\  \sup_{0< r \leq R^{1+\varepsilon}}\  \frac{\log N_{r} \Big( \b(x, R) \cap F \Big)}{\log R - \log r}.
 \]
The reason for the inclusion of $\varepsilon$ in the above definition is to avoid the supremum taking undesirably high values for $r \approx R$.  Our interpretation of $A_F^\varepsilon (x,R)$ is the `best guess' for the dimension of $F$ at location $x$ and scale $R$.

The above definition is motivated by the Assouad dimension; a commonly studied and important notion of dimension used to study fractals. Indeed, the Assouad dimension was introduced in the 1970s by Patrice Assouad during his doctoral studies, motivated by questions involving embeddability of metric spaces into Euclidean space \cite{assouadphd, assouad}.  Unsurprisingly, the Assouad dimension highlights the \emph{extreme} dimensional properties of the set in question and, roughly speaking, returns an exponent representing the \emph{most difficult} location and scale at which to cover the set.   The \emph{Assouad dimension} of a non-empty subset $F$ of $X$, $\dim_\text{A} F$, is defined by
\[
\dim_\text{A} F \ = \  \inf \left\{ \alpha \  : \   (\exists C>0) \,  (\forall 0<r<R\leq 1) \,  (\forall x \in F) \  N_r\big( \b(x,R) \cap F \big) \ \leq \ C (R/r)^\alpha \right\}.
\]
Note that whether one uses boxes or balls (or any other reasonable family of $R$-scale objects) to define $C(x,R)$ does not alter the definition. For more information on the Assouad dimension, including basic properties, we refer the reader to \cite{luk, robinson, fraser_assouad}. A related notion is the \emph{upper box dimension}, which can be defined by
\[
\overline{\dim}_\text{B} F \ = \  \inf \left\{ \alpha \  : \   (\exists C>0) \,  (\forall 0<r\leq 1) \  N_r\big(  F \big) \ \leq \ C (1/r)^\alpha \right\}.
\]
This is clearly related to the Assouad dimension, but the key difference is that it is not sensitive to location and gives an overall average dimension rather than the supremum of a localised quantity.  It is always true that $\overline{\dim}_\text{B} F \leq \dim_\text{A} F$, with equality occurring if the set is sufficiently homogeneous: at all scales and locations the set is no more difficult to cover than it is on average over the whole set.  For the sets we consider in this paper the upper box dimension will be equal to the related lower box dimension, in which case we just write $\dim_\text{B} F$ for the common value. 

It is easy to see that
\begin{equation} \label{quasiassouad}
\lim_{\varepsilon \to 0} \, \limsup_{R \to 0} \, \sup_{x \in F} A_F^\varepsilon (x,R) \ \leq \ \dim_\text{A} F,
\end{equation}
but in most cases, including the sets we consider in this paper, one has equality and even
\[
\limsup_{R \to 0} \, \sup_{x \in F} A_F^\varepsilon (x,R) \ = \ \dim_\text{A} F
\]
for sufficiently small $\varepsilon$. In fact the quantity on the left hand side of (\ref{quasiassouad}) defines the quasi-Assouad dimension, introduced in \cite{quasiassouad}.

\subsection{Large Deviations}
Since (quasi-)Assouad dimension is determined by the largest $A_F^\varepsilon (x,R)$, the $\eps$-local Assouad dimension at scale $R$, can be over \emph{all} $x$, it is a natural question to know how much this quantity varies as $x\in F$ varies.  So to consider this question quantitatively, we consider, for an appropriate Borel probability measure $\mathbb{P}$ on $F$ and $\lambda >0$ (or in a suitable range), 
\begin{equation} \label{measureestimate666}
\mathbb{P} \ \Big( \big\{ x \in F \  : \ A_F^\varepsilon (x,R)>\lambda \big\} \Big).
\end{equation}
If our set is homogeneous, then $A_F^\varepsilon (x,R)$ is essentially independent of $x$ (the Assouad and box dimensions coincide), so the quantity in \eqref{measureestimate666} is interesting when there is some inhomogeneity.   A, by now classic, and elementary, example of an inhomogeneous limit set is provided by Bedford-McMullen carpets, and for this first quantitative analysis of inhomogeneity, we focus on this class.
Indeed, our main theorem shows that the  quantity in \eqref{measureestimate666} satisfies a Large Deviation Principle as $R \to 0$, with a rate function which we give explicitly.  We then give examples of limit sets $E$ and $F$ for which $\dim_\text{A}E-\dim_\text{B}E>\dim_\text{A}F-\dim_\text{B}F$, but the inhomogeneity can be seen to be stronger for $F$ via the rate of large deviations.

We also prove a Central Limit Theorem (CLT) for the variable $A_F^\varepsilon (x,R)$.  This is less directly connected to Assouad dimension since, the set of points being measured here are not the set of points at which the Assouad dimension is `achieved'.  However, this does give some further information on inhomogeneity.  We note that in the case of measures of balls, compared to the local dimension, a CLT was proved in \cite{LepSau12} for certain classes of non-conformal dynamical systems similar to those considered here.

The ideas in this paper extend to a wide class of examples of inhomogeneous fractals.  We note that, for interesting results, it is important that the inhomogeneity varies across the fractal. This means that, for example, the large deviations for limit sets of Mandelbrot percolation will be degenerate.  We would expect Large Deviation principles for general self-affine sets, including the more general carpets introduced in \cite{baranski, feng, fraser, lalley}, higher dimensional analogues \cite{kenyon, sponges}, and generic self-affine sets in the setting of Falconer \cite{affine}, for example.   It would be particularly interesting, and technically challenging, to consider genuinely `parabolic' systems, related to \cite{HuVai09}, where the canonical measures would be expected to give polynomial large deviations.

\section{Our  class of $(\times m, \times n)$-invariant sets}

We will study the type of large deviations problem discussed in the previous section for a class of $(\times m, \times n)$-invariant sets, i.e. compact subset of the 2-torus which are invariant under the endomorphism $f_{m,n}: (x,y) \to (mx \mod 1, ny \mod 1)$ for $m,n \in \mathbb{N}$ with $2\leq m<n$.  First we introduce a symbolic coding of the torus which we identify with $[0,1]^2$ in the natural way.  For $m$ and $n$ as above, let $\mathcal{I} = \{0, \dots, m-1\}$ and $\mathcal{J} = \{0,\dots, n-1\}$ and divide $[0,1]^2$ into a uniform $m \times n$ grid and label the $mn$ resultant subrectangles by $\mathcal{I} \times \mathcal{J}$ in the natural way counting from bottom left to top right.  Choose a subset of the rectangles $\mathcal{D} \subseteq \mathcal{I} \times \mathcal{J}$ of size at least 2 and for each $d = (i,j) \in \mathcal{D}$, associate a contraction $T_d: [0,1]^2 \to [0,1]^2$ defined by $T_d(x,y) = (x/m+i/m, \, y/n+j/n)$ and observe that this is the inverse branch of $f_{m,n}$ restricted to the interior of the rectangle corresponding to $d$ and then extended to the boundary by continuity.  Let $\mathcal{D}^\infty = \{ \mathbf{d} = (d_1, d_2, \dots) : d_l = (i_l,j_l) \in \mathcal{D} \}$ be the set of infinite words over $\mathcal{D}$ and let $\Pi :  \mathcal{D}^\infty \to [0,1]^2$ be the natural projection defined by
\[
\Pi(\mathbf{d}) \ =  \ \bigcap_{l \in \mathbb{N}} \, T_{d_1} \circ \cdots \circ  T_{d_l} \big([0,1]^2\big).
\]
This map gives us our symbolic coding, but we note that some points may not have a unique code.  Write $F = \Pi(\mathcal{D}^\infty)$ and observe that $F$ is both forward and backward invariant under $f_{m,n}$ and, moreover, is a dynamical repeller for $f_{m,n}$.  Such non-conformal repellers have been studied extensively in a variety of contexts, for example, in the theory of iterated function systems $F$ is a self-affine set in the class introduced independently by Bedford \cite{bedford} and McMullen \cite{mcmullen}.  In particular, it is the unique non-empty compact set satisfying
\[
F \ = \ \bigcup_{d \in \mathcal{D}} T_d(F).
\]
Bedford and McMullen computed the Hausdorff and box dimensions of these sets and several years later Mackay \cite{mackay} computed the Assouad dimension, see also \cite{fraser_assouad}. For $i \in \mathcal{I}$, let $C_i = \lvert \{(i',j') \in \mathcal{D} : i'=i\} \rvert$ be the number of chosen rectangles in the $i$th column and let $C_{\max} = \max_{i \in \mathcal{I}} C_i$.  Also let $\pi : \mathcal{I} \times \mathcal{J} \to \mathcal{I}$ be the projection onto the first coordinate. 
\begin{thm}[Bedford-McMullen-Mackay] \label{carpetdimensions}
The Assouad and box dimensions of $F$ are given by 
\[
\dim_\text{\emph{A}} F \ =  \  \frac{\log \lvert\pi\mathcal{D} \rvert }{ \log m} \, + \,  \frac{ \log C_{\max}}{ \log n}
\]
and
\[
\dim_\text{\emph{B}} F \ =  \  \frac{\log \lvert\pi\mathcal{D} \rvert }{ \log m} \, + \,  \frac{ \log(\lvert\mathcal{D} \rvert/ \lvert \pi\mathcal{D} \rvert )}{ \log n}.
\]
\end{thm}

Observe that the term $\log \lvert\pi\mathcal{D} \rvert /\log m$ appears in both formulae.  This is the dimension of the self-similar projection of $F$ onto the first coordinate. The second term relates to the dimension of fibres: for the Assouad dimension, $\log C_{\max}/ \log n$ is the maximal fibre dimension (also given by a self-similar set) and for the box dimension, the second term can be interpreted as an `average fibre dimension', observing that $\lvert\mathcal{D} \rvert/ \lvert \pi\mathcal{D} \rvert$ is the arithmetic average of the $C_i$.  The box and Assouad dimensions coincide if and only if $C_i$ is the same for all $i \in \pi \mathcal{D}$, commonly referred to as the \emph{uniform fibres case}.  In this case our subsequent analysis is not very interesting and so from now on we will assume that the $C_i$ are not uniform.  We will not use the box dimension directly, but the value it attains will serve a role in our analysis.  As such we refer the reader to \cite[Chapter 3]{falconer} for more details and basic properties of box dimension.

\subsection{Canonical measures on $F$}
The canonical measures supported on the sets $F$ are Bernoulli measures and are defined as follows.  Let $\{p_{d}\}_{d \in \mathcal{D}}$ be a strictly positive probability vector and $\mu$ be the corresponding Bernoulli measure on $\mathcal{D}^\infty$.  Also, let $\{p_{i}\}_{i \in \pi \mathcal{D}}$ be the corresponding `projected weights', where $p_i = \sum_{d \in \pi^{-1}i} p_d$ is the total weight of the $i$th column.  Finally, let $\mathbb{P} = \mu \circ \Pi^{-1}$ be the push forward of $\mu$, which is a Borel probability measure supported on $F$.

The choice of  $\{p_{d}\}_{d \in \mathcal{D}}$ depends on which measure one considers to be canonical for a particular study.  If one wants the measure of maximal entropy, then $p_d = 1/|\mathcal{D}|$ for all $d \in \mathcal{D}$.  If one wants the measure of maximal (Hausdorff) dimension, then we consider the McMullen measure, with weights:
\[
p_d = \frac{C_{\pi(d)}^{\log m/ \log n -1}}{m^s}
\]
where $s$ is the Hausdorff dimension.  Another natural choice is a measure which is uniform on the column structure, i.e. any choice of $\{p_d\}_{d \in \mathcal{D}}$ satisfying $p_i = 1/|\pi \D|$ (for all $i \in \pi \mathcal{D}$). The push forward of such a measure under $\pi$ is the measure of maximal entropy for the `base map'.  If the measure is uniform on columns and then within each column one distributes the weight evenly among the allowable rectangles, then one obtains a measure which `realises the Assouad dimension', in the sense of Konyagin and Volberg \cite{konyagin} (using approximate squares as balls), see \cite{sponges} for the details.  Moreover, this is the only known invariant probability measure which realises the Assouad dimension.

We stress that we cover \emph{all} Bernoulli measures in our analysis, thus catering to all tastes.  Note that if one wanted to consider Lebesgue measure, one could study the set of $|\D|^k$ depth $k$ rectangles with side $m^{-k}\times n^{-k}$ intersected with the set of interest at level $k$.  The `escape' of Lebesgue measure as the depths increase would interfere with the limit laws close to the mean, but may be dominated by our variable of interest if we are sufficiently far from the mean.

\begin{figure}[H]
	\centering
	\includegraphics[width=160mm]{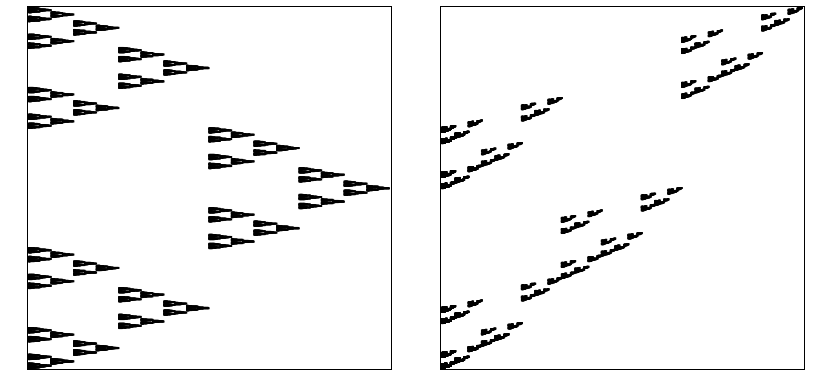}
	\caption{Two examples from our class of self-affine carpets. The example on the left has $m=2$, $n=3$, $C_0=2$, $C_1=1$ and the example on the right has $m=3$, $n=4$, $C_0=2$, $C_1=1$, $C_2=1$.}
\end{figure}

\section{Key technical covering lemmas}

Let $\mathbf{d} \in \mathcal{D}^\infty$ and $r>0$ be small.  Define $l_1(r), l_2(r)$ to be the unique natural numbers satisfying
\[
m^{-l_1(r)} \,  \leq \, r \, < \,  m^{-l_1(r)+1}
\]
and
\[
n^{-l_2(r)} \,  \leq \, r \, < \,  n^{-l_2(r)+1}.
\]
The approximate square `centred'  at $\mathbf{d} = ((i_1, j_1), (i_2, j_2), \dots)  \in \mathcal{D}^\infty$ with `radius' $r>0$ is defined by
\[
Q(\mathbf{d}, r) \ = \ \Big\{ \mathbf{d}' = ((i'_1, j'_1), (i'_2, j'_2), \dots)  \in \mathcal{D}^\infty : i'_l = i_l \ \forall l \leq l_1(r), \text{ and }  j'_l = j_l \ \forall  l \leq l_2(r)  \Big\},
\]
and the geometric projection of this set, $\Pi\big(Q(\mathbf{d}, r)\big)$, is a subset of $F$ which contains $\Pi(\mathbf{d})$ and naturally sits inside a rectangle which is `approximately a square' in that it has a base with length $m^{-l_1(r)} \in (r/m, r]$ and height $n^{-l_2(r)} \in (r/n, r]$.  These sets will play the role of $\b(x, R)$ when we apply our notion of local Assouad dimension. That is, for $\mathbf{d} \in \mathcal{D}^\infty$,  $R>0$ and $\varepsilon>0$ we define:
\[
A_{\mathcal{D}^\infty}^\varepsilon (\mathbf{d},R) \ :=\  \sup_{0< r \leq R^{1+\varepsilon}}\  \frac{\log N_{r} \big( \Pi\big(Q(\mathbf{d}, r)\big)  \big)}{\log R - \log r}.
 \]

For $\mathbf{d} = ( (i_1, j_1), (i_2, j_2), \dots) \in \mathcal{D}^\infty$, $i \in \mathcal{I}$ and $s, t \in \mathbb{N}$ with $s<t$, let
\[
\mathcal{P}_\mathbf{d}^i( s,t ) \ = \ \frac{\lvert \{ i_l = i : l = s, \dots, t\}\rvert}{t-s+1}
\]
be the proportion of time the symbol $i$ is used in the base component of $\mathbf{d}$ between positions $s$ and $t$.  Write 
\[
\mathcal{C}_\mathbf{d}(R) \ = \ \prod_{i \in \mathcal{I}} C_{i}^{ \mathcal{P}_\mathbf{d}^i( l_2(R)+1,l_1(R) ) }
\]
and
\[
\mathcal{C}_\mathbf{d}(r,R) \ = \ \prod_{i \in \mathcal{I}} C_{i}^{ \mathcal{P}_\mathbf{d}^i( l_2(R)+1,l_2(r) ) }
\]
for the geometric average of the $C_i$ chosen according to $\mathbf{d}$ in the ranges $l_2(R)+1, \dots, l_1(R)$ and $ l_2(R)+1, \dots, l_2(r)$ respectively, noting that the first range only depends on $R$.  We can now state our key covering lemma.

\begin{lma} \label{keylemma} Let $0<r<R < 1$ and $\mathbf{d} \in \mathcal{D}^\infty$.  Suppose $l_2(r) \geq  l_1(R)$, then 
\begin{eqnarray*}
\frac{\log N_{r} \big( \Pi\big(Q(\mathbf{d}, R)\big) \big)}{\log R - \log r}  &=&  \frac{\bigg( \frac{\log (\lvert \mathcal{D} \rvert/\mathcal{C}_\mathbf{d}(R) )}{\log m} \, + \,  \frac{\log \mathcal{C}_\mathbf{d}(R)}{\log n}\bigg) \log R   \ -  \  \bigg( \frac{\log \lvert\pi\mathcal{D} \rvert }{ \log m} \, + \,  \frac{ \log(\lvert\mathcal{D} \rvert/ \lvert \pi\mathcal{D} \rvert )}{ \log n}  \bigg) \log r}{\log R - \log r} \\ \\
&\,& \qquad \qquad \qquad \qquad + \ \frac{ O(1)}{\log(R/r)}
\end{eqnarray*}
and if $l_2(r) <  l_1(R)$, then
\begin{eqnarray*}
\frac{\log N_{r} \big( \Pi\big(Q(\mathbf{d}, R)\big) \big)}{\log R - \log r} &=& \frac{\log \lvert\pi\mathcal{D} \rvert }{ \log m} \ +\ \frac{\log \mathcal{C}_\mathbf{d}(r,R)}{\log n} \ + \ \frac{ O(1)}{\log(R/r)}.
\end{eqnarray*}
Moreover, the implied constants in both $O(1)$ terms are independent of $\mathbf{d} \in \mathcal{D}^\infty$.
\end{lma}

\begin{proof}
Let $\mathbf{d} = ((i_1, j_1), (i_2, j_2), \dots)  \in \mathcal{D}^\infty$ and first suppose $l_2(r) \geq  l_1(R)$.  To prove this covering result, we rely on the following geometric observations.  When one looks at an approximate square $\Pi\big(Q(\mathbf{d}, R)$ one finds that it is made up of several horizontal strips with base length the same as that of the approximate square, i.e. $O(R)$, and height
\[
n^{-l_1(R)}
\]
which is considerably smaller, but by assumption still larger than $r$.  These strips are images of $F$ under $l_1(R)$-fold compositions of maps $\{T_d\}_{d \in \mathcal{D}}$  and so to keep track of how many there are, one counts rectangles in the appropriate columns for each map in the composition.  The total number is seen to be:
\[
\prod_{l=l_2(R)+1}^{l_1(R)} C_{i_l}.
\]
We want to cover these strips by sets of diameter $r$ and so we iterate the construction inside each horizontal strip until the height is $O(r)$. Indeed, this takes $l_2(r) - l_1(R)$ iterations and, this time, for every iteration we pick up $\lvert \mathcal{D} \rvert$ smaller rectangles, rather than just those inside a particular column.  At this stage we are left with a large collection of rectangles with height $O(r)$ and base somewhat larger.  We now iterate inside each of these rectangles until we obtain a family of rectangles each with base of length $O(r)$.  This takes a further $l_1(r)-l_2(r)$ iterations.  We can then cover the resulting collection by small sets of diameter $O(r)$, observing that different sets formed by this last stage of iteration can be covered simultaneously and so for each of the last iterations we only require a factor of $\lvert\pi\mathcal{D} \rvert$ more covering sets. Putting all these estimates together, yields
\begin{eqnarray*}
N_{r} \big( \Pi\big(Q(\mathbf{d}, R)\big) \big) &\asymp&  \Bigg(\prod_{l=l_2(R)+1}^{l_1(R)} C_{i_l} \Bigg) \ \Bigg( \lvert \mathcal{D} \rvert^{l_2(r) - l_1(R)} \Bigg) \ \Bigg(  \lvert\pi\mathcal{D} \rvert^{l_1(r) - l_2(r)} \Bigg) \\ \\
&=&  \Bigg(\prod_{i \in \mathcal{I}} C_{i}^{ \mathcal{P}_\mathbf{d}^i( l_2(R)+1,l_1(R) ) }\Bigg)^{l_1(R) - l_2(R)} \  \lvert \mathcal{D} \rvert^{l_2(r) - l_1(R)}  \  \lvert\pi\mathcal{D} \rvert^{l_1(r) - l_2(r)} \\ \\
&\asymp&  \mathcal{C}_\mathbf{d}(R)^{\log R /\log n - \log R / \log m} \  \lvert \mathcal{D} \rvert^{\log R /\log m - \log r / \log n}  \  \lvert\pi\mathcal{D} \rvert^{\log r /\log n - \log r / \log m} \\ \\
&=&  R^{\log \mathcal{C}_\mathbf{d}(R) / \log n + \log (\lvert \mathcal{D} \rvert/\mathcal{C}_\mathbf{d}(R) ) /\log m}\  r^{ \log(\lvert\pi\mathcal{D} \rvert/ \lvert \mathcal{D} \rvert )/ \log n \, - \, \log \lvert\pi\mathcal{D} \rvert / \log m}.
\end{eqnarray*}
Taking logs and dividing by $\log(R/r)$ yields
\begin{eqnarray*}
\frac{\log N_{r} \big( \Pi\big(Q(\mathbf{d}, R)\big) \big)}{\log R - \log r}  &=&  \frac{\bigg( \frac{\log (\lvert \mathcal{D} \rvert/\mathcal{C}_\mathbf{d}(R) )}{\log m} \, + \,  \frac{\log \mathcal{C}_\mathbf{d}(R)}{\log n}\bigg) \log R   \ -  \  \bigg( \frac{\log \lvert\pi\mathcal{D} \rvert }{ \log m} \, + \,  \frac{ \log(\lvert\mathcal{D} \rvert/ \lvert \pi\mathcal{D} \rvert )}{ \log n}  \bigg) \log r}{\log R - \log r} \\ \\
&\,& \qquad \qquad + \ \frac{O(1)}{\log(R/r)}
\end{eqnarray*}
as required, noting that the implied constants in $O(1)$ are independent of $\mathbf{d} \in \mathcal{D}^\infty$.

Secondly, suppose $l_2(r) <  l_1(R)$.  We proceed as before, but this time one obtains a family of horizontal strips with height $O(r)$ after $l_2(r)$ step, which is before the bases become smaller than $O(R)$.  The effect is that the `middle term' above (concerning powers of $\lvert \mathcal{D} \rvert$) is not required.  The horizontal strips are then covered as before yielding
\begin{eqnarray*}
N_{r} \big( \Pi\big(Q(\mathbf{d}, R)\big) \big) &\asymp&  \Bigg(\prod_{l=l_2(R)+1}^{l_2(r)} C_{i_l} \Bigg) \ \Bigg(  \lvert\pi\mathcal{D} \rvert^{l_1(r) - l_1(R)} \Bigg) \\ \\
&=&  \Bigg(\prod_{i \in \mathcal{I}} C_{i}^{ \mathcal{P}_\mathbf{d}^i( l_2(R)+1,l_2(r) ) }\Bigg)^{l_2(r) - l_2(R)}  \  \lvert\pi\mathcal{D} \rvert^{l_1(r) - l_1(R)} \\ \\
&\asymp&  \mathcal{C}_\mathbf{d}(r,R)^{\log R /\log n - \log r / \log n}   \  \lvert\pi\mathcal{D} \rvert^{\log R /\log m - \log r / \log m} \\ \\
&=&  (R/r)^{\log \mathcal{C}_\mathbf{d}(r,R) / \log n + \log\lvert \pi D \rvert /\log m}.
\end{eqnarray*}
Taking logs and dividing by $\log(R/r)$ yields
\begin{eqnarray*}
\frac{\log N_{r} \big( \Pi\big(Q(\mathbf{d}, R)\big) \big)}{\log R - \log r} &=& \frac{\log \lvert\pi\mathcal{D} \rvert }{ \log m} \ +\ \frac{\log \mathcal{C}_\mathbf{d}(r,R)}{\log n} \ + \ \frac{O(1)}{\log(R/r)}
\end{eqnarray*}
as required, again noting that the implied constants in $O(1)$ are independent of $\mathbf{d} \in \mathcal{D}^\infty$.
\end{proof}

\section{Our modified family of observables}

Since we are interested in small scales $r$ and $R$ where $r\le R^{1+\eps}$, and thus $R/r$ is large,  it makes sense to replace
\[
\frac{\log N_{r} \big( \Pi\big(Q(\mathbf{d}, R)\big) \big)}{\log R - \log r} 
\]
with the obvious asymptotic formula gleaned from Lemma \ref{keylemma}. As such, let $A(\mathbf{d}, R, r)$ be defined piecewise by
\[
A(\mathbf{d}, R, r) \ = \ \frac{\bigg( \frac{\log (\lvert \mathcal{D} \rvert/\mathcal{C}_\mathbf{d}(R) )}{\log m} \, + \,  \frac{\log \mathcal{C}_\mathbf{d}(R)}{\log n}\bigg) \log R   \ -  \  \bigg( \frac{\log \lvert\pi\mathcal{D} \rvert }{ \log m} \, + \,  \frac{ \log(\lvert\mathcal{D} \rvert/ \lvert \pi\mathcal{D} \rvert )}{ \log n}  \bigg) \log r}{\log R - \log r} 
\]
for $0 < r <R^{\log n/\log m}$ and
\[
A(\mathbf{d}, R, r) \ = \ \frac{\log \lvert\pi\mathcal{D} \rvert }{ \log m} \ +\ \frac{\log \mathcal{C}_\mathbf{d}(r,R)}{\log n}
\]
for $R^{\log n/\log m} \leq  r <R$, observing that the phase transition between the two types of behaviour seen in Lemma \ref{keylemma} occurs at the `critical scale' $r=R^{\log n/\log m}$, where both formulae reduce to
\[
A(\mathbf{d}, R, r) \ =\ \frac{\log \lvert\pi\mathcal{D} \rvert }{ \log m} \ +\ \frac{\log \mathcal{C}_\mathbf{d}(R)}{\log n}.
\]
Since $R$ is our fixed reference scale, we want to understand how $A(\mathbf{d}, R, r) $ varies with $r$ for fixed $R$ and $\mathbf{d}$.   For values of $r$ smaller than $R^{\log n/\log m}$, the `column average' term $\mathcal{C}_\mathbf{d}(R)$ does not depend on $r$ and this makes the asymptotic formula easier to analyse.  Lemma \ref{keylemma} yields
\[
A(\mathbf{d}, R, r)   \ \to \ \frac{\log \lvert\pi\mathcal{D} \rvert }{ \log m} \, + \,  \frac{ \log(\lvert\mathcal{D} \rvert/ \lvert \pi\mathcal{D} \rvert )}{ \log n} \ = \ \dim_\text{B} F
\]
as $r \to 0$. Moreover, elementary optimisation shows that there are three possible `shapes' for $A(\mathbf{d}, R, r) $ in this region.  If $\mathcal{C}_\mathbf{d}(R) > \lvert\mathcal{D} \rvert/ \lvert \pi\mathcal{D} \rvert $, i.e. the observed average is greater than the actual average, then $\dim_\text{B} F < A(\mathbf{d}, R, r)  \leq \dim_\text{A} F$ for $r = R^{\log n/\log m}$, but then $A(\mathbf{d}, R, r)$ strictly decreases to the box dimension as $r \to 0$.  Similarly, if $\mathcal{C}_\mathbf{d}(R) < \lvert\mathcal{D} \rvert/ \lvert \pi\mathcal{D} \rvert $, i.e. the observed average is less than the actual average, then $A(\mathbf{d}, R, r)  < \dim_\text{B} F$ for $r = R^{\log n/\log m}$, but then $A(\mathbf{d}, R, r)$ strictly increases to the box dimension as $r \to 0$.  Finally, if $\mathcal{C}_\mathbf{d}(R) = \lvert\mathcal{D} \rvert/ \lvert \pi\mathcal{D} \rvert $, then $A(\mathbf{d}, R, r) $ is constantly equal to $\dim_\text{B} F$ for $r$ less than the critical scale.

For values of $r$ greater than the critical scale, $A(\mathbf{d}, R, r) $ is more difficult to analyse, despite the apparently simpler formula.  For example, it is not continuous (in $r$) with discontinuities potentially occurring whenever $l_2(r)$ has a discontinuity. 

\begin{figure}[H]
	\centering
	\includegraphics[width=\textwidth]{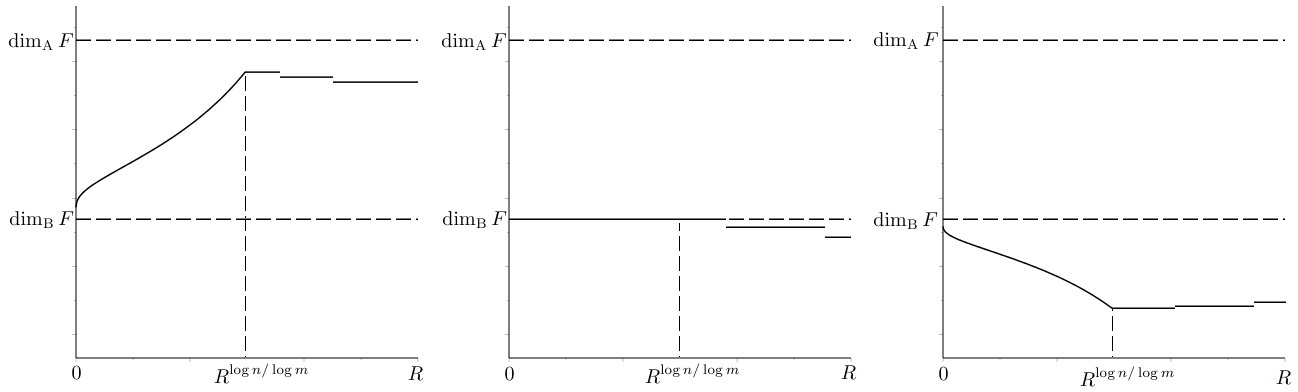}
	\caption{Three plots of $A(\mathbf{d}, R, r)$ as functions of $r$.  (Here $m=2$, $n=3$ and $R =0.3$.) }
\end{figure}

We now wish to study statistical properties of
\[
\sup_{0< r \leq R^{1+\varepsilon}}\  A(\mathbf{d}, R, r)
\]
and, in particular, how it behaves around the maximal value possible, $\dim_\text{A} F$.  It is sufficient to let $R$ tend to zero through an exponential sequence of scales and so to simplify notation slightly from now on we consider $R=n^{-k}$ for $k \in \mathbb{N}$, which renders $l_2(R) = k$. Also, let $\gamma = \log n/\log m$. It follows that
\begin{eqnarray*}
A^\varepsilon_0(\mathbf{d}, k) \ & :=& \  \sup_{0< r \leq n^{-(1+\varepsilon)k}}\  A(\mathbf{d}, n^{-k}, r) \\ \\
  &=  &  \max \left\{  \dim_\text{B} F, \  \sup_{n^{-k \log n/ \log m} \leq r \leq n^{-(1+\varepsilon)k}}\  \frac{\log \lvert\pi\mathcal{D} \rvert }{ \log m} \ +\ \frac{\log \mathcal{C}_\mathbf{d}(r,n^{-k})}{\log n} \right\} \\ \\
  & = &  \max \left\{  \dim_\text{B} F, \ \frac{\log \lvert\pi\mathcal{D} \rvert }{ \log m} \ + \  \max_{L \in \{\lceil (1+\varepsilon )k\rceil, \dots,\lceil \gamma k\rceil\}}\frac{ \sum_{i \in \mathcal{I}}\mathcal{P}_\mathbf{d}^i(k+1, L)  \log C_{i} }{\log n}  \right\}.
\end{eqnarray*}
This reduction is possible since for any $k \in \mathbb{N}$ 
\[
\sup_{0< r \leq n^{-k\log n/ \log m}}\  A(\mathbf{d}, n^{-k}, r)  \ = \ \max \left\{ \dim_\text{B} F, \  \frac{\log \lvert\pi\mathcal{D} \rvert }{ \log m} \ + \   \frac{ \sum_{i \in \mathcal{I}}\mathcal{P}_\mathbf{d}^i(k+1, \lceil \gamma k\rceil)  \log C_{i} }{\log n}    \right\}
\]
by the discussion above.  In particular, for $r$ in this range,  the supremum of $A(\mathbf{d}, n^{-k}, r) $ is obtained either as $r$ approaches 0, in which case it is the box dimension, or when $r$ is the critical scale $n^{-k\gamma}$, in which case it is the second value.

We will study the statistical properties of the observables $A^\varepsilon_0(\mathbf{d}, k)$ in the subsequent two sections.  For now, note that for all $\mathbf{d} \in \mathcal{D}^\infty$ we have
\[
 A_{\mathcal{D}^\infty}^\varepsilon (\mathbf{d},n^{-k})  \ =  \ A_0^\varepsilon(\mathbf{d}, k) \  + \ \frac{ O(1)}{k \varepsilon}
\]
(as $k \to \infty$) with the implied constants independent of $\mathbf{d} \in \mathcal{D}^\infty$.  This follows from Lemma \ref{keylemma} and the above discussion.  Moreover, for all $\mathbf{d} \in \mathcal{D}^\infty$ and $k \in \mathbb{N}$
\[
0 \  < \ \dim_\text{B} F \ \leq \  A_0^\varepsilon(\mathbf{d}, k) \ \leq \  \dim_\text{A} F
\]
and
\[
  \lim_{k \to \infty} \, \sup_{\mathbf{d} \in \mathcal{D}^\infty} \, A_0^\varepsilon(\mathbf{d}, k) \ = \  \dim_\text{A} F.
\]

\section{Large deviations results}

Let
\begin{equation}
\alpha \ := \ \frac{\log \lvert\pi\mathcal{D} \rvert }{ \log m} \ +\  \frac{ \sum_{i \in \mathcal{I}} p_i  \log C_{i} }{\log n} \ < \ \dim_\text{A} F
\label{eq:alpha}
\end{equation}

be the relevant mean. As we are interested in how inhomogeneous $F$ is, we let $\max\{\alpha, \dim_\text{B} F\} \ <\  \lambda \ < \  \dim_\text{A} F$, and $0< \varepsilon < \log n/\log m -1$ and, similarly to \eqref{measureestimate666}, consider
\begin{equation} \label{measureestimate2}
\mu \ \Big( \big\{\mathbf{d} \in \mathcal{D}^\infty \  : \ A_0^\varepsilon(\mathbf{d}, k)>\lambda \big\} \Big).
\end{equation}
It will turn out that this decays exponentially in $k$ and so we obtain a Large Deviation Principle.

By the definition of $A_0^\varepsilon(\mathbf{d}, k)$, (\ref{measureestimate2}) is equal to
\begin{equation} \label{measureestimate3}
\mu \ \left( \left\{\mathbf{d} \in \mathcal{D}^\infty \   : \   \max_{L \in \{\lceil (1+\varepsilon )k\rceil, \dots,\lceil \gamma k\rceil\}} \sum_{i \in \mathcal{I}}\mathcal{P}_\mathbf{d}^i(k+1, L)  \log C_{i}   \ > \  \left(\lambda - \frac{ \log \lvert\pi\mathcal{D} \rvert}{ \log m} \right) \log n  \right\}   \right).
\end{equation}

To simplify notation, write
\[
\lambda' \ =  \ \left(\lambda - \frac{ \log \lvert\pi\mathcal{D} \rvert}{ \log m} \right) \log n 
\]
 and
\[
S_L \  =  \ \frac{1}{L} \sum_{l=1}^{L} \log C_{i_l}.
\]
Using this notation and the fact that $\log C_{i_l}$ is an i.i.d. process we can rewrite  (\ref{measureestimate3}) as
\[
\mu \ \Big( \big\{\mathbf{d} \in \mathcal{D}^\infty \   : \  \max_{L \in \{\lceil \varepsilon k \rceil, \dots, \lceil (\gamma-1) k \rceil \}} S_L > \lambda' \big\}  \Big).
\]

We know from standard large deviations theory (see for example \cite[Section 9]{Bil12}) that there are positive constants $C_1$ and $C_2$ such that for all $L \in \mathbb{N}$
\[
 C_1 \exp\big(-L \cdot I(\lambda')\big)  \ \leq \  \mu \ \Big( \big\{\mathbf{d} \in \mathcal{D}^\infty \   : \    S_L > \lambda' \big\}  \Big) \ \leq \  C_2 \exp\big(-L \cdot I(\lambda')\big) 
\]
where $I(\cdot)$ is the classical rate function which is given by
\[
I(\lambda') = - \lim_{N \to \infty} \frac{1}{N} \log \mu \ \Big( \big\{\mathbf{d} \in \mathcal{D}^\infty \   : \    S_N > \lambda' \big\} \Big).
\]
It is clear that $I(\lambda')=0$ for $\lambda' \leq \sum_{i \in \mathcal \pi D} p_i  \log C_i$ and $I(\lambda') = +\infty$ for $\lambda' \geq \log C_{\max}$. Moreover, since we know the distribution of $\log C_{i_l}$,  $I$ is the Legendre transform of the cumulant generating function
\[
\theta \mapsto \log \mathbb{E}(\exp(\theta \log C_{i_l})) = \log \sum_{i \in \mathcal \pi D} p_i C_i^\theta.
\]
for values of $\lambda'$ in the open interval between these boundary values.  In particular, this means that in this range it is strictly increasing, continuous and convex. 

We have 
\begin{eqnarray*}
\mu \ \Big( \big\{\mathbf{d} \in \mathcal{D}^\infty \   : \   \max_{L \in \{\lceil \varepsilon k \rceil, \dots, \lceil (\gamma-1) k \rceil \}} S_L> \lambda' \big\}  \Big) &\geq&  \mu \ \Big( \big\{\mathbf{d} \in \mathcal{D}^\infty \   : \  S_{\lceil \varepsilon k \rceil} > \lambda' \big\}  \Big) \\ \\
&\geq&   C_1 \exp(- \lceil \varepsilon k \rceil  I(\lambda')) 
\end{eqnarray*}
and the union bound yields
\begin{eqnarray*}
\mu \ \Big( \big\{\mathbf{d} \in \mathcal{D}^\infty \   : \   \max_{L \in \{\lceil \varepsilon k \rceil, \dots, \lceil (\gamma-1) k \rceil \}} S_L> \lambda' \big\}  \Big) & \leq & \sum_{L=  \lceil \varepsilon k \rceil}^{ \lceil (\gamma-1) k \rceil } \mu \ \Big( \big\{\mathbf{d} \in \mathcal{D}^\infty \   : \    S_L > \lambda' \big\}  \Big) \\ \\
 & \leq & \sum_{L=  \lceil \varepsilon k \rceil}^{ \lceil (\gamma-1) k \rceil } C_2 \exp(-L \cdot I(\lambda')) \\ \\
 & \leq &    C_2 \frac{\exp(- \lceil \varepsilon k \rceil I(\lambda'))}{1-\exp(-I(\lambda'))}.
\end{eqnarray*}
Recalling that the mean of our observable is $\alpha$ defined in \eqref{eq:alpha}, the following large deviations result now follows immediately: 

\begin{thm} \label{LDthm}
For $\max\{\alpha, \dim_\text{\emph{B}} F\} \leq \lambda < \dim_\text{\emph{A}} F$ and $0 < \varepsilon < \log n/ \log m-1$, we have 
\[
\lim_{k \to \infty} \  \frac{1}{k} \, \log  \mu  \, \Big( \big\{\mathbf{d} \in \mathcal{D}^\infty \  : \ A_0^\varepsilon(\mathbf{d}, k)>\lambda \big\} \Big) \ = \ - \varepsilon \,  I\left(  \lambda' \right) \ = \ - \varepsilon \,  I\left( \left(\lambda - \frac{ \log \lvert\pi\mathcal{D} \rvert}{ \log m} \right) \log n \right).
\]
\end{thm}
Note that $I\big(\lambda' \big)$ is increasing in $\lambda$ and
\[
I\big(\lambda' \big) \searrow 0  \qquad (\lambda \searrow \alpha)
\]
and
\[
I\big(\lambda' \big)  \ \nearrow  \  -\log \Bigg( \sum_{i \, : \,  C_i = C_{\max}} p_i \Bigg) \, >0 \qquad  (\lambda \nearrow \dim_\text{A} F).
\]
Interestingly, if $\dim_\text{B} F>\alpha$, then the rate function does not approach 0 at the left hand side of its domain; since the domain is restricted by $\max\{\alpha, \dim_\text{B} F\}$.   The limit in Theorem \ref{LDthm} is zero for all $\lambda$ smaller than this value and so if $\dim_\text{B} F>\alpha$, then the large deviations has a discontinuity at $\lambda=\dim_\text{B} F$.  Indeed, if $\mu$ is such that the weights $\{p_i\}_{i \in \pi \mathcal{D}}$ are uniform, i.e. all equal to $1/\lvert \pi \mathcal{D} \rvert$, then $\alpha$ is strictly smaller than the box dimension by virtue of the arithmetic-geometric mean inequality.

\subsection{An alternative formulation} \label{formulation2}
Of course, the above large deviations result can also be expressed in terms of an observable based on points $x \in F$, rather than $\mathbf{d}  \in \mathcal{D}^\infty$, and arbitrary scales $R \to 0$, rather than the sequence $n^{-k}$.  This is what was originally proposed in the introduction. We have to be slightly careful here if $x$ is coded by multiple $\mathbf{d} \in \mathcal{D}^\infty$, but the number of codes for a given $x$ can be no more than 4 and it can be shown that the set of $x$ with non-unique codes is of $\mathbb{P}$-measure zero.  One way of defining such an observable  in our setting is:
\[
A_F^\eps(x, R) \ := \ \max\{  A_{\mathcal{D}^\infty}^\eps(\mathbf{d}, R) : \Pi(\mathbf{d}) = x\}.
\]
Alternatively, if one wanted an observable of exactly the same form as that in the introduction, then one could choose $C(x,R)$ to be the approximate $R$-square centered at $\mathbf{d} \in \mathcal{D}^\infty$, where $\mathbf{d}$ is the lexicographically minimal code for $x$, for example.  The following result can be obtained for either version.
\begin{thm}
For $\max\{\alpha, \dim_\text{\emph{B}} F\} \leq \lambda < \dim_\text{\emph{A}} F$ and $0 < \varepsilon < \log n/ \log m-1$, we have
\[
\lim_{R \to 0} \  \frac{\log  \mathbb{P}  \, \Big( \big\{ x \in F  \  : \ A_F^\eps(x, R) >\lambda \big\} \Big)}{\log R}   \  = \ \frac{\varepsilon \,  I\left(  \lambda' \right)}{\log n} \ = \  \frac{\varepsilon \,  I\left( \left(\lambda - \frac{ \log \lvert\pi\mathcal{D} \rvert}{ \log m} \right) \log n \right)}{\log n}.
\]
\end{thm}

\begin{proof}
This is a straightforward corollary of Theorem \ref{LDthm} and we omit the details.  The important points are that $A_{\mathcal{D}^\infty}^\eps(\mathbf{d}, n^{-k})-A_0^\eps(\mathbf{d}, k) =O(1/k)$ (with the implied constants independent of $\mathbf{d} \in \mathcal{D}^\infty$) and the rate function is continuous.   Also, since we are taking logs, letting $R$ tend to zero through the exponential sequence of scales $R=n^{-k}$ $(k \to \infty)$ does not affect the convergence.  
\end{proof}

At first sight, it may seem a little strange that the extra factor of $1/\log n $ appears in this formulation of our result since it looks as though it comes from our choice of sequence $R=n^{-k} \to 0$.  Of course the same result is obtained by letting $R \to 0$ through \emph{any} sequence and in fact the factor of $1/\log n $ appears because $l_2(R) \sim -\log R/\log n$.

\subsection{Examples}

We will now consider two pairs of examples.  Throughout this section we fix $m=3$ and $n=4$ and $\mu$ will always be a Bernoulli measure which assigns uniform weight to the columns (1/3 in every case).  

 We will now describe the first pair of examples.  For the first carpet let $C_0=3$, $C_1=2$,  $C_2=2$ and for the second carpet, let $C_0=4$, $C_1=1$,  $C_2=1$.  For the first example
\[
\dim_\text{A} F \ =  \  1+ \frac{\log 3}{\log 4} \ \approx \ 1.792
\]
and
\[
\dim_\text{B} F   \ =  \  1+ \frac{\log (7/3)}{\log 4} \ \approx \ 1.611.
\]
For the second example
\[
\dim_\text{A} F \ =  \  2
\]
and
\[
\dim_\text{B} F   \ =  \  1+ \frac{\log (6/3)}{\log 4} \ = 1.5.
\]
Thus, if we measure inhomogeneity  in the crude fashion alluded to in the introduction by considering the difference between the Assouad and box dimensions, then we deduce that the second example is more inhomogeneous than the first.  However, our large deviations analysis yields finer information. The rate function $\lambda \mapsto I(\lambda')$ is plotted on the left hand side of Figure \ref{plotsfig} below, with the first example a solid line and the second example a dashed line.  We observe that although the first example was `less inhomogeneous' by the crude measure, the exponential rate of decay in the large deviations result is much larger for first example in the range 1.72 to 1.79 indicating a greater degree of inhomogeneity when viewed at that `dimension'.

 We will now describe the second pair of examples.  For the first example, let  $C_0=3$, $C_1=3$, $C_2=1$ and for the second example, let $C_0=3$, $C_1=2$, $C_2=2$.  For both examples 
\[
\dim_\text{A} F \ =  \  1+ \frac{\log 3}{\log 4} \ \approx \ 1.792,
\]
and
\[
\dim_\text{B} F   \ =  \  1+ \frac{\log (7/3)}{\log 4} \ \approx \ 1.611.
\]
but the large deviations are rather different.  The rate function $\lambda \mapsto I(\lambda')$ is plotted on the right hand side of Figure \ref{plotsfig} below, with the first example a solid line and the second example a dashed line.  Note that the rate functions approach different limits as $\lambda \to \dim_\text{A} F$.  This is because for the first example there are two columns which realise the Assouad dimension ($i=0$ or $1$) and so the rate function approaches $-\log(2/3)$ instead of $-\log(1/3)$, which is the limit for the second example.  This can be interpreted as the first example being more homogeneous near the Assouad dimension because there are more locations in the fractal which return values close to the maximum.

\begin{figure}[H]  \label{plotsfig}
\setlength{\unitlength}{\textwidth}
\begin{picture}(1,0.47)
\put(0.02,0){\includegraphics[trim=3cm 15cm 8cm 0cm, width=0.45\textwidth]{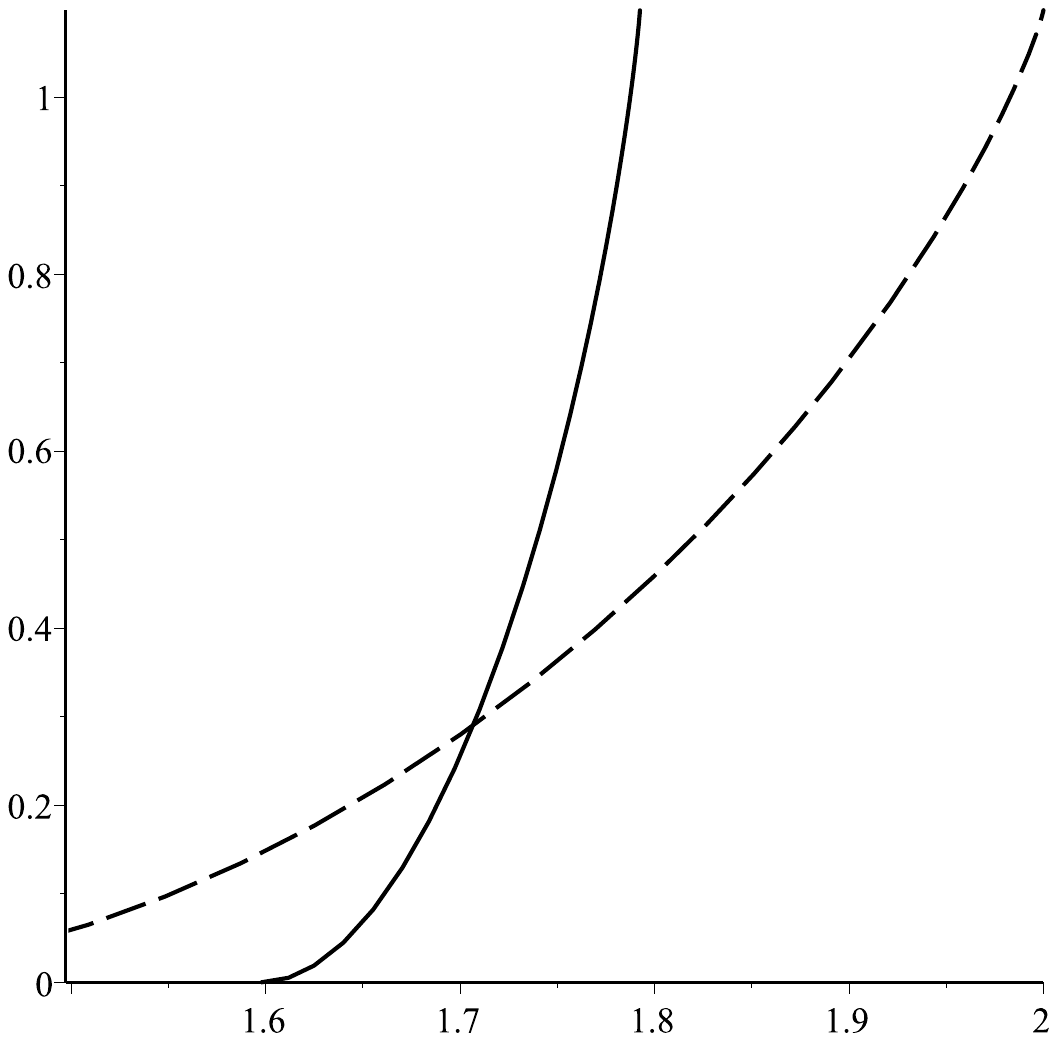}}
\put(0.53,0){\includegraphics[trim=3cm 15cm 8cm 0cm, width=0.45\textwidth]{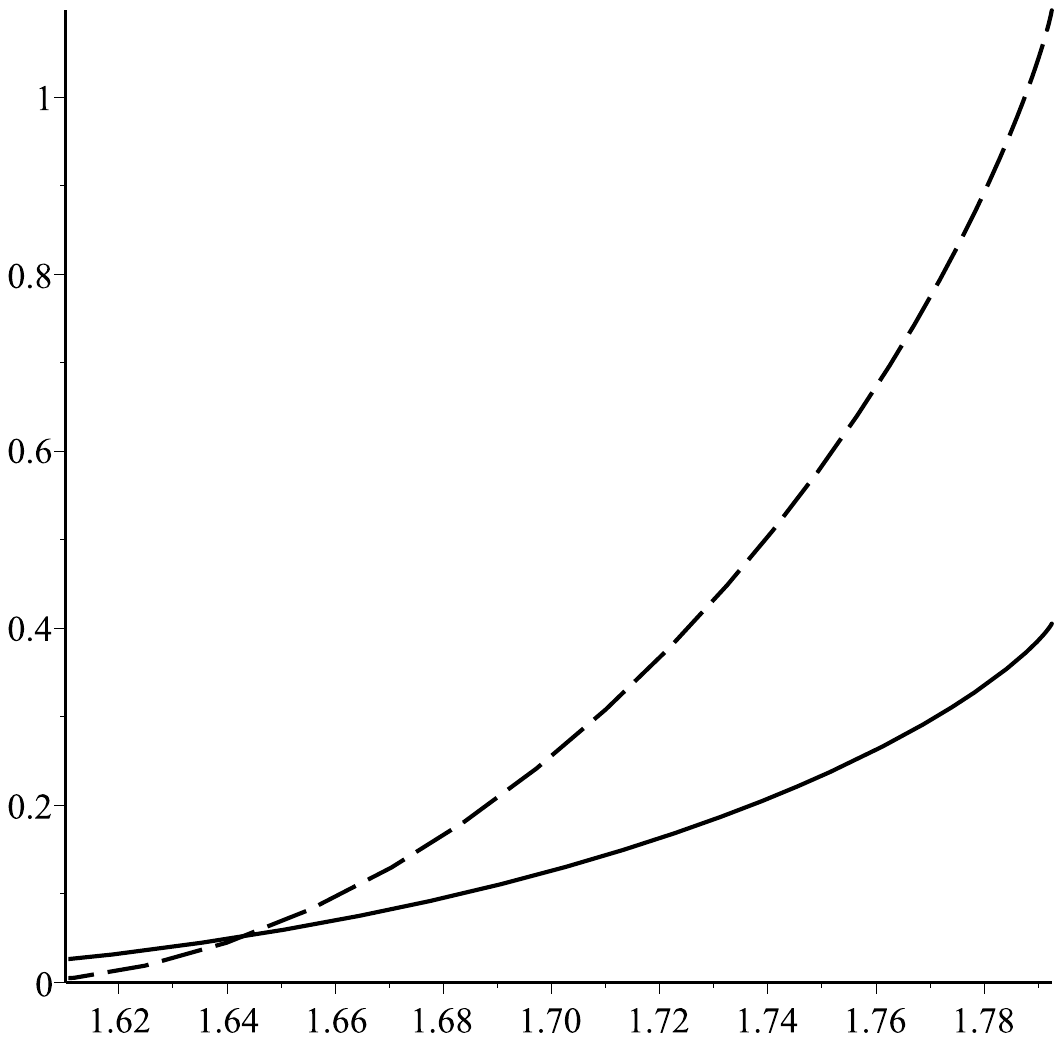}}
\end{picture}
\caption{Two plots of $\lambda \mapsto I(\lambda')$ in the relevant ranges.  Note that none of the functions approach zero at the left hand side of their domain.}
\end{figure}

\section{A Central Limit Theorem for a single scale}

In the previous section we saw that the large deviations for the `maximal average' process were dominated by the first term.  As such, in this section we consider the limiting behaviour one gets when only considering that term, i.e. without taking a maximum. Fix $0< \varepsilon < \log n/\log m -1$,  let $\delta = 1+\varepsilon$ and consider
\[
A^\delta(\mathbf{d}, k) \ : = \    \frac{\log \lvert\pi\mathcal{D} \rvert }{ \log m} \ +\  \frac{\sum_{i \in \mathcal{I}}\mathcal{P}_\mathbf{d}^i(k, \lceil \delta k \rceil)  \log C_{i} }{\log n} .
\]

We know from the previous section that for $\mu$ almost all $\mathbf{d}$ we have
\[
A^\delta(\mathbf{d}, k) \ \to \   \frac{\log \lvert\pi\mathcal{D} \rvert }{ \log m} \ +\  \frac{ \sum_{i \in \mathcal{I}} p_i  \log C_{i} }{\log n} \ = \  \alpha \ \ (k \to \infty),
\]
but we can say more about the distribution around the limit.  Note that $\alpha$ is independent of $\varepsilon$ (and thus $\delta$). We are interested in finding a sequence of functions $\alpha_k: \mathbb{R} \to \mathbb{R}$ such that
\begin{equation} \label{measureestimate}
\mu \ \Big( \big\{\mathbf{d} \in \mathcal{D}^\infty \  : \ A^\delta(\mathbf{d}, k)>\alpha - \alpha_k(\tau)\big\} \Big)
\end{equation}
converges to a non-degenerate function of $\tau$ as $k \to \infty$.  By the definition of $A^\delta(\mathbf{d}, k)$ and the formula for $\alpha$ above, (\ref{measureestimate}) is equal to
\[
\mu \ \Big( \big\{\mathbf{d} \in \mathcal{D}^\infty \   : \   \sum_{i \in \mathcal{I}}\mathcal{P}_\mathbf{d}^i(k, \lceil \delta k \rceil )  \log C_{i}  > \Big(\sum_{i \in \mathcal{I}} p_i  \log C_{i}\Big) - \alpha_k(\tau) \log n \big\} \Big).
\]
Observe that
\[
S_k \ :=  \ \sum_{i \in \mathcal{I}}\mathcal{P}_\mathbf{d}^i(k, \lceil \delta k \rceil )  \log C_{i}  \ = \ \frac{1}{\lceil \delta k \rceil - k + 1} \sum_{l=k}^{\lceil \delta k \rceil} \log C_{i_l},
\]
and since $\mu$ is a Bernoulli measure, the events $\log C_{i_l}$ are i.i.d. according to $\{p_i\}_{i \in \pi \mathcal{D}}$ with mean
\[
c: = \sum_{i \in \pi \mathcal{D}} p_i \log C_i \ \in \ [0, \log C_{\max} ]
\]
and variance
\[
\sigma := \sum_{i \in \pi \mathcal{D}} p_i (\log C_i - c)^2 \  > \  0.
\]
Thus the classical Central Limit Theorem (see for example \cite[Section 27]{Bil12}) yields
\[
\frac{\sqrt{\lceil \delta k \rceil - k + 1} \, (c-S_k)}{\sqrt{\sigma} }  \Rightarrow \Phi,
\]
where $\Phi$ denotes the Normal Distribution and $\Rightarrow$ denotes convergence in distribution.  Define $\alpha_k: \mathbb{R} \to \mathbb{R}$ by
\[
\alpha_k(\tau) \ := \  \frac{\tau \sqrt{\sigma}}{(\log n)\sqrt{\lceil \delta k \rceil - k + 1}}
\]
and note that
\[
S_k >  \Big(\sum_{i \in \mathcal{I}} p_i  \log C_{i}\Big) - \alpha_k(\tau) \log n  = c - \alpha_k(\tau) \log n
\]
if and only if
\[
\tau > \frac{\sqrt{\lceil \delta k \rceil - k + 1}\,  (c-S_k)}{\sqrt{\sigma}}.
\]
We therefore have the following Central Limit Theorem:
\begin{thm} \label{CLTthm} For $\tau \in \mathbb{R}$
\[
\lim_{k \to \infty}\  \mu \ \Big( \big\{\mathbf{d} \in \mathcal{D}^\infty \  : \ A^\delta(\mathbf{d}, k)> \alpha-\alpha_k(\tau)\big\} \Big) \ = \ \Phi(\tau).
\]
\end{thm}

We note that although the set of points being (asymptotically) measured here are not those at which Assouad dimension is `achieved', this nevertheless can give a measure of how inhomogeneous the set $F$ is.

Clearly if we assume $p_d>0$ for all $d \in \mathcal{D}$ (which is require for $\mu$ to have $F$ as its support) and $C_i$ is not the same for every $i \in \pi \mathcal{D}$, then $\alpha$ is strictly less than the Assouad dimension of $F$.  However, $\alpha$ may be larger or smaller than the box dimension; recall the discussion in the previous section. 

\subsection{An alternative formulation}

As with the large deviations result, we can provide an alternative formulation of the Central Limit Theorem from the previous section.  For $R>0$ and $x \in F$, let
\[
A_F^\delta(x, R) \ : = \    \frac{\log \lvert\pi\mathcal{D} \rvert }{ \log m} \ + \  \max_{\mathbf{d} \in \mathcal{D}^\infty : \Pi(\mathbf{d}) = x} \frac{\sum_{i \in \mathcal{I}}\mathcal{P}_\mathbf{d}^i(l_2(R), \lceil \delta l_2(R) \rceil)  \log C_{i} }{\log n} 
\]
and
\[
\alpha_R(\tau) \ := \  \frac{\tau \sqrt{\sigma}}{(\log n)\sqrt{\lceil \delta l_2(R) \rceil - l_2(R) + 1}}.
\]
The following result can be obtained in a similar fashion to Theorem \ref{CLTthm}:

\begin{thm} For $\tau \in \mathbb{R}$
\[
\lim_{R  \to 0}\  \mathbb{P} \ \Big( \big\{ x \in F \  : \ A_F^\delta(x, R)> \alpha-\alpha_R(\tau)\big\} \Big) \ = \ \Phi(\tau).
\]
\end{thm}

\section*{Acknowledgments}

MT would like to thank the University of Manchester, where much of this work was done, for their hospitality.  He is also grateful to F. P\`ene for a useful conversation.  JMF thanks the University of St Andrews for their hospitality during several recent research visits involving this work.

\end{document}